\documentclass[12pt,a4paper]{article}
\usepackage{ucs}
\usepackage[applemac]{inputenc}
\usepackage[T1]{fontenc}
\usepackage[english]{babel}		
\usepackage{
	amsmath,
	amsthm,
	amssymb,
	amsfonts,
	mathtools
}
\usepackage[dvips]{hyperref}
\hypersetup{	
	colorlinks,
	linkcolor=black,
	citecolor=black,
	urlcolor=black,
	breaklinks=true
}
\DeclareMathAlphabet{\mathpzc}{OT1}{pzc}{m}{it}
\usepackage{geometry}
\newtheoremstyle{lemma}{\topsep}{\topsep}
	{\itshape}
	{}
	{\bfseries}
	{.}
	{\newline}
	{\thmname{#1}\thmnumber{ #2}\thmnote{ #3}}	
\theoremstyle{lemma}
	\newtheorem{lemma}{Lemma}[section]
	\newtheorem{proposition}[lemma]{Proposition}
	\newtheorem{corollary}[lemma]{Corollary}
	\newtheorem{theorem}[lemma]{Theorem}

\newtheoremstyle{definition}{\topsep}{\topsep}
	{}
	{}
	{\bfseries}
	{.}
	{\newline}
	{\thmname{#1}\thmnumber{ #2}\thmnote{ #3}}	
\theoremstyle{definition}
	\newtheorem{definition}[lemma]{Definition}
	\newtheorem{remark}[lemma]{Remark}
	\newtheorem{example}[lemma]{Example}

\newcommand{\he}{\ensuremath{\alpha}}

\newcommand{\R}{\ensuremath{\mathbb{R}}}

\newcommand{\HM}{\ensuremath{\mathcal{H}}}

\newcommand{\M}{\ensuremath{\mathcal{M}}}

\DeclarePairedDelimiter\norm{\lVert}{\rVert}

\newcommand{\dHM}{\ensuremath{\,\text{d}\HM^{\he}}}

\renewcommand{\phi}{\varphi}
\renewcommand{\epsilon}{\varepsilon}

\hypersetup{
	pdftitle={A characterisation of inner product spaces by the maximal circumradius of spheres},
	pdfsubject={},
	pdfauthor={Sebastian Scholtes},
	pdfkeywords={},
	pdfcreator={},
	pdfproducer={}
}

\makeindex
\begin{document}

\title{A characterisation of inner product spaces by the maximal circumradius of spheres}
\author{\href{mailto:sebastian.scholtes@rwth-aachen.de}{Sebastian Scholtes}}
\date{\today}
\maketitle

\begin{abstract}
	We give a new characterisation of inner product spaces amongst normed vector spaces in terms of the maximal cirumradius of spheres. It turns out that a 
	normed vector space $(X,\norm{\cdot})$ with $\dim X\geq 2$ is an inner product space if and only if all spheres are not degenerate, i.e. the maximal circumradius 
	of points on the sphere equals the radius of the sphere or more formally 
	 \begin{align*}
	 	\sup_{\substack{u,v,w\in \partial B_{r}(x_{0})\\u\not=v\not=w\not=u}}r(u,v,w)=r\quad\text{for all }x_{0}\in X\text{ and all }r>0.
	\end{align*}
\end{abstract}
\centerline{\small Mathematics Subject Classification (2000): 46C15; 46B20}
\bigskip

It is an important and particularly well researched topic to find conditions that classify inner product spaces amongst, say, mere normed vector spaces. This area of research
really took off in 1935, where several of these characterisations were given. Best-know of these characterisations is the one by Jordan and von Neumann \cite{Jordan1935a}.
They showed that an inner product, which induces the norm, can be defined in a normed vector space if and only if the parallelogram law holds. For detailed historical
information we refer the reader to the introduction and chronological publication list in \cite{Amir1986a}. A good overview of the criteria available in their respective times 
can be found in \cite{Day1962a,Istratescu1987a} and, of course, the very comprehensive \cite{Amir1986a}. 
The book \cite{Alsina2010a} gives a survey of the many characterisations in terms of norm derivatives.
Up to the present day this topic still remains very active and fruitful, 
so that, because of the sheer number of publications in this area, we restrain ourselves from trying to give a survey and instead refer the interested reader to the books mentioned above and the papers which are cited below as a starting point.\\

In what follows we give a, to the best of our knowledge, new characterisation of inner product spaces amongst normed vector spaces, which involves the maximal
circumradius of points in spheres. In a metric space $(X,d)$ the 
\emph{circumradius} $r(u,v,w)$ of three mutually distinct points $u,v,w\in X$ is defined to be the circumradius of the triangle constituted by the isometric embedding of these three
points in the Euclidean plane and can be written as
\begin{align*}
	r(u,v,w)=\frac{abc}{\sqrt{(a+b+c)(a+b-c)(a-b+c)(-a+b+c)}}=\frac{abc}{\sqrt{-D(u,v,w)}},
\end{align*}
where $a\vcentcolon=d(u,v)$, $b\vcentcolon=d(v,w)$, $c\vcentcolon=d(w,u)$ and incidentally $D(u,v,w)$ is the \emph{Cayley-Menger determinant}, 
see \cite[40, pp. 97--99, especially Ex. 3]{Blumenthal1970a}. If $D(u,v,w)=0$ we set $r(u,v,w)=\infty$. This quantity plays an important role in 
geometric curvature energies and ideal knots, where various curvature energies are defined in terms of $r(u,v,w)$. The most established of these energies are probably
the \emph{thickness} $\Delta$ and the \emph{integral $p$-Menger curvature} $\M_{p}$ of a metric space $(X,d)$, defined by
\begin{align*}
	\Delta[X]\vcentcolon=\inf_{\substack{u,v,w\in X\\u\not= v\not= w\not= u}}r(u,v,w),\quad
	\M_{p}(X)\vcentcolon=\int_{X}\int_{X}\int_{X}\frac{1}{r(u,v,w)^{p}}\dHM(u)\dHM(v)\dHM(w).
\end{align*}
Concerning these curvature energies we want to refer the reader to \cite{Gonzalez1999a} for the thickness, 
and \cite{Hahlomaa2008a,Strzelecki2010a,Scholtes2012a} and the references therein for the integral Menger curvature and ``intermediate energies''.
The notions above were also the cause why the author ``accidentally'' found the present characterisation of inner product spaces. We want to
mention the role of integral Menger curvature for $p=2$ in the solution of the Painlev\'e problem, i.e. to find geometric characterisations of removable sets for 
bounded analytic functions, see \cite{Pajot2002a} for a detailed presentation and references.\\

In \cite{Klee1960a} a characterisation of inner product spaces under normed vector spaces in terms of the \emph{$X$-radius}
\begin{align*}
	r_{X}(M)\vcentcolon= \inf\{r\in (0,\infty)\mid \exists x_{0}\in X: M\subset \overline B_{r}(x_{0})\}
\end{align*}
of a bounded subset $M$ of a normed vector space $X$ were given. If $\dim X\geq 3$ the existence of an inner product is equivalent to $r_{C}(C)=r_{X}(C)$ for every
bounded convex set $C\subset X$. Further results connected with the classic circumradius of points can be found in \cite{Alsina1996a}, where a characterisation of 
inner product spaces is given in terms of the radius of the inscribed circle and the circumradius of triangles being comprised by two independent vectors. 
More precisely they use two notions of each radius: one notion is defined only in terms of distances and the other additionally involves an inner product. 
Both these notions agree in inner product spaces. To make sense of the second notion for normed vector spaces, a generalized inner product on normed vector 
spaces, namely the norm derivative, is used. Now it can be shown that inner product spaces are exactly those normed vector spaces in which these two notions of radius agree. 
In case of the inscribed circle this was restricted to spaces of dimension larger than two and for the circumradius to spaces of dimension at least three with the
additional restriction of being strictly convex. Later on in \cite{Tomas2005a} one of these authors showed that a strictly convex normed vector space of dimension larger 
than three is an inner product space if and only if a circumcenter, in the sense of the intersection of the perpendicular bisectors of the edges, 
exists for certain types of triangles. In general there may be more than one point in this intersection, i.e. more than one circumcenter. Subsequently it was shown that a 
strictly convex normed vector space is an inner product space, if and only if for every collection of points on a sphere centred at the origin one of these circumcenters is the origin.\\

The paper is organized as follows: After looking into the behaviour of points not fulfilling the parallelogram law we shortly prove a well-known characterisation of inner
product spaces by Wilson involving the Euclidean four point property. We then briefly remark on a space of four points not embeddable in inner product spaces and
aspects of geometry in $(\R^{2},\norm{\cdot}_{\infty})$ in general. After that we give the characterisation of inner product spaces announced in the title.\\

\textbf{Acknowledgement}\\
The author whishes to thank his advisor Heiko von der Mosel for constant support and encouragement.

\section{A characterisation of inner product spaces in terms of the maximal circumradius of spheres}

\begin{definition}[(Inner product space, non inner product space)]
	Let $(X,\norm{\cdot})$ be a normed real vector space. We call $(X,\norm{\cdot})$ an \emph{inner product space}, or short \emph{ips}, if there exists an inner product 
	$\langle \cdot,\cdot\rangle$ on $X$, which induces the norm, i.e. $\norm{x}=\sqrt{\langle x,x \rangle}$ for all $x\in X$. If there exists no such inner
	product on $X$, we call $(X,\norm{\cdot})$ a \emph{non inner product space}, or short \emph{nips}. We refer to a sphere $\partial B_{r}(x_{0})$ as
	being \emph{degenerate} if there are points $u,v,w\in \partial B_{r}(x_{0})$, such that $r(u,v,w)>r$.
\end{definition}

We start by investigating what characterises points for which the parallelogram law does not hold.

\begin{lemma}[(Properties of points not fulfilling the parallelogram law)]\label{noparallelogramlawthentherearefourpointsnotembeddable}
	Let $(X,\norm{\cdot}_{X})$ be a normed vector space and $u,v\in X$ points which do not satisfy the parallelogram law, i.e.
	\begin{align}\label{noparallelogramlawforuandv}
		2\norm{u}_{X}^{2}+2\norm{v}_{X}^{2}\not= \norm{u+v}_{X}^{2}+ \norm{u-v}_{X}^{2}.
	\end{align}
	Then $(\{0,u,v,-v\},d_{X})$, where $d_{X}$ is the metric $d_{X}(a,b)=\norm{a-b}_{X}$, is not isometrically embeddable in any inner product space.
\end{lemma}
\begin{proof}
	\textbf{Step 1}
		From (\ref{noparallelogramlawforuandv}) we know that $u\not\in\{0,v,-v\}$ and $v\not\in\{0,u,-u\}$. Assume that there exists an isometric embedding
		$\phi: (\{0,u,v,-v\},d_{X})\to (V,\norm{\cdot}_{V})$ into an inner product space.
		Without loss of generality, by translating by $-\phi(0)$ if necessary, we can assume $\phi(0)=0$. We now have
		\begin{align*}
			\norm{\phi(u)}_{V}=\norm{\phi(u)-\phi(0)}_{V}=d_{V}(\phi(u),\phi(0))=d_{X}(u,0)=\norm{u-0}_{X}=\norm{u}_{X}
		\end{align*}
		and by the same argument $\norm{\phi(v)}_{V}=\norm{v}_{X}=\norm{-v}_{X}=\norm{\phi(-v)}_{V}$. The parallelogram law for $v$ and $-v$ gives us
		\begin{align*}
			\MoveEqLeft 4\norm{v}_{X}^{2}=2\norm{v}_{X}^{2}+2\norm{v}_{X}^{2}=  2\norm{\phi(v)}_{V}^{2}+2\norm{\phi(-v)}_{V}^{2}\\
			&= \norm{\phi(v)+\phi(-v)}_{V}^{2}+ \norm{\phi(v)-\phi(-v)}_{V}^{2}\\
			&= \norm{\phi(v)+\phi(-v)}_{V}^{2}+ d_{V}(\phi(v),\phi(-v))^{2}\\
			&= \norm{\phi(v)+\phi(-v)}_{V}^{2}+ d_{X}(v,-v)^{2}\\
			&= \norm{\phi(v)+\phi(-v)}_{V}^{2}+ \norm{v-(-v)}_{X}^{2}\\
			&= \norm{\phi(v)+\phi(-v)}_{V}^{2}+ 4\norm{v}_{X}^{2},
		\end{align*}
		so that $\norm{\phi(v)+\phi(-v)}_{V}=0$ and consequently $\phi(-v)=-\phi(v)$.\\
	\textbf{Step 2}
		This gives us
		\begin{align*}
			\MoveEqLeft 2\norm{u}_{X}^{2}+2\norm{v}_{X}^{2}=2\norm{\phi(u)}_{V}^{2}+2\norm{\phi(v)}_{V}^{2}\\
			&=\norm{\phi(u)+\phi(v)}_{V}^{2}+\norm{\phi(u)-\phi(v)}_{V}^{2}\\
			&=\norm{\phi(u)-\phi(-v)}_{V}^{2}+\norm{\phi(u)-\phi(v)}_{V}^{2}\\
			&=\norm{u-(-v)}_{X}^{2}+\norm{u-v}_{X}^{2}
		\end{align*}
		which contradicts (\ref{noparallelogramlawforuandv}).
\end{proof}

We now have the means to prove en passant the most interesting implication, here (iv) implies (i), in a variation of a well-known theorem of 
Wilson \cite[12. Theorem]{Wilson1932a}, for the special case of normed vector spaces.

\begin{proposition}[(Characterisation of ips via Euclidean four point property)]\label{characterisationofinnerproductspaces}
	Let $(X,\norm{\cdot}_{X})$ be a normed vector space. Then the following are equivalent
	\begin{enumerate}
		\item
			$(X,\norm{\cdot}_{X})$ is an inner product space,
		\item
			there exists an inner product space $(V,\norm{\cdot}_{V})$ such that $(X,\norm{\cdot}_{X})$ is isometrically embeddable in $(V,\norm{\cdot}_{V})$,
		\item
			all sets $F=\{u,v,w,x\}\subset X$ are isometrically embeddable in $(\R^{3},\norm{\cdot}_{2})$,
		\item
			for all sets $F=\{u,v,w,x\}\subset X$ exists an inner product space $(V(F),\norm{\cdot}_{V(F)})$ in which $F$ is isometrically embeddable.
	\end{enumerate}
\end{proposition}
\begin{proof}
	Clearly (i) implies (ii) and (iii) implies (iv). Now for (ii) implies (iii). By (ii) a given set $\{u,v,w,x\}\subset X$ is isometrically mapped to four points in 
	the inner product space $V$ and therefore also to four points in the completion of this inner product space, which again is an inner product space. 
	This means (iii) is simply a consequence of the observation that the Euclidean four-point property of \cite{Wilson1932a}, see also \cite[(e4pp-0), p.116]{Day1962a},
	holds for inner product spaces.
	The remaining direction is the contrapositive of Lemma \ref{noparallelogramlawthentherearefourpointsnotembeddable}, 
	since a normed vector space is an inner product space	if and only if the parallelogram law holds for all $u,v\in X$.
\end{proof}

It is kind of surprising that although any three points in a pseudometric space can be isometrically embedded in the plane, there are four points in a Banach space,
in the example we will present even the center and three points on the boundary of a ball, which cannot be isometrically embedded in any inner product space. 
A similar space and the nonexistence of isometric embeddings in Euclidean spaces can be found in \cite{Robinson2006a}. We refer the reader to \cite{Yang1984a} for 
more complicated examples and conditions on sets which cannot be isometrically embedded in any Hilbert space.

\begin{example}[(Four points not isometrically embeddable in ips)]\label{fourpointsnotembeddable}
	Consider the Banach space $(\R^{2},\norm{\cdot}_{\infty})$ and 
	\begin{align*}
		\theta\vcentcolon=(0,0), u\vcentcolon=(0,1), v\vcentcolon=(1,0), -v=(-1,0)\in\R^{2}.
	\end{align*}
	Then according to Lemma \ref{embeddingthreepointsonasphereandthecenter} for any inner product space $(V,\langle \cdot,\cdot\rangle)$ there exists no, 
	possibly nonlinear, isometry $\phi: (\{\theta,u,v,-v\},\norm{\cdot}_{\infty})\to (V,\langle \cdot,\cdot\rangle)$, since one can easily check that $r(u,v,-v)=\infty$.
	Alternatively one can verify that the parallelogram law does not hold for $u$ and $v$ and then use Lemma \ref{noparallelogramlawthentherearefourpointsnotembeddable}.
\end{example}

In Lemma \ref{circumradiusofspheresinnorminfty} in the appendix we continue this example and show that we find points of arbitrary circumradius on any sphere 
in $(\R^{2},\norm{\cdot}_{\infty})$.

\begin{lemma}[(Embedding of three points on a sphere and their center)]\label{embeddingthreepointsonasphereandthecenter}
	Let $(X,d)$ be a metric space and $x_{0},u,v,w\in X$ mutually distinct points, such that $d(x_{0},x)=r$ and $r>0$ for all $x\in\{u,v,w\}$. Then $\{x_{0},u,v,w\}$ 
	is isometrically embeddable in $3$-dimensional Euclidean space -- which incidentally is the same as being isometrically embeddable in any inner product space --
	if and only if $r(u,v,w)\leq r$.
\end{lemma}
\begin{proof}
	Let $r(u,v,w)<\infty$, or else see Case 3.
	Without loss of generality we start by embedding the three points $\{u,v,w\}$ isometrically in the $xy$ plane in $\R^{3}$ with a mapping $\phi$, such that
	the the center of the circumcircle of $\phi(x),\phi(y),\phi(z)$ is the origin; and have to investigate wether we can extend $\phi$ isometrically to $x_{0}$, i.e.
	$\phi(u),\phi(v),\phi(w)\in \partial B_{r}(\phi(x_{0}))$.\\	
	\textbf{Case 1}
		If $r(u,v,w)\leq r$ we can place $\phi(x_{0})=(0,0,\sqrt{r^{2}-r(u,v,w)^{2}})$ in the appropriate distance along the $z$-axis.\\
	\textbf{Case 2}
		If $r<r(u,v,w)$ the distance $r$, where $\phi(x_{0})$ would have to be placed, if we expect our embedding to be isometrically, is too short to reach the $z$-axis.\\
	\textbf{Case 3}
		If $r(u,v,w)=\infty$ the mutually distinct points $\phi(u),\phi(v),\phi(w)$ lie on a straight line. In Euclidean space a line can meet a sphere at most two 
		times, see Lemma \ref{intersectionofsphereandline}, such that there is no $a\in \R^{3}$ with $\phi(u),\phi(v),\phi(w)\in \partial B_{r}(a)$.
\end{proof}

\begin{lemma}[(All spheres in non inner product spaces are degenerate)]\label{allnonipsspheresaredegenerate}
	Let $(X,\norm{\cdot})$ be a normed vector space, which is a non inner product space. Then for all $x_{0}\in X$ and all $r>0$ there exist mutually distinct 
	$u,v\in \partial B_{r}(x_{0})$ such that $r(u,v,-v)>r$.
\end{lemma}
\begin{proof}
	Without loss of generality we can assume $x_{0}=0$, because we can always translate by $-x_{0}$.
	By scaling the two vectors in the parallelogram law with $r$, we know from \cite[Theorem 2.1]{Day1947a} that for each $r>0$ exist vectors $u,v\in \partial B_{r}(0)$, 
	such that
	\begin{align*}
		\norm{u+v}^{2}+\norm{u-v}^{2}\not=2\norm{u}^{2}+2\norm{v}^{2}=4r^{2},
	\end{align*}
	so that by Lemma \ref{noparallelogramlawthentherearefourpointsnotembeddable} the points $\{0,u,v,-v\}$ are not isometrically embeddable in any inner product space.
	Considering Lemma \ref{embeddingthreepointsonasphereandthecenter} we have $r<r(u,v,-v)$.
\end{proof}

\begin{lemma}[(Supremum is attained)]\label{supremumofcircumradiusisattained}
	Let $(V,\norm{\cdot})$ be an inner product space, $\dim V= 2$. Then for $x_{0}\in V$, $r>0$ and all pairwise distinct
	$u,v,w\in \partial B_{r}(x_{0})$ we have $r(u,v,w)=r$.
\end{lemma}
\begin{proof}
	Since $\partial B_{r}(x_{0})=x_{0}+\partial B_{r}(0)$ and $r(u,v,w)=r(u+x_{0},v+x_{0},w+x_{0})$ we can safely assume $x_{0}=0$. 
	Let $P\subset V$ be a 
	two-dimensional subspace. Then $P$ is a two-dimensional inner product space, which is isometrically isomorphic to the 
	Euclidean plane $(\R^{2},\norm{\cdot}_{2})$, as
	can easily been seen by choosing an orthonormal base $(b_{1},b_{2})$ of $P$ and noting that 
	$\norm{\alpha b_{1}+\beta b_{2}}=\sqrt{\alpha^{2}+\beta^{2}}=\norm{(\alpha,\beta)}_{2}$, which gives rise to an isometric isomorphism $\phi$.
	Since the circumradius only depends on the distances of points we have $r(x,y,z)=r(\phi(x),\phi(y),\phi(z))=r$.
\end{proof}

We should remark that the condition $\dim X\geq 2$, we will impose on our space, is not really a restriction, as every one-dimensional 
space is an inner product space, because the norm is essentially a positive constant times the absolute value.

\begin{definition}[(The quantity $S(M)$)]
	To shorten notation we denote
	\begin{align*}
		S(M)\vcentcolon=\sup_{\substack{u,v,w\in M\\u\not= v\not= w\not= u}}r(u,v,w)\in (0,\infty]
	\end{align*}
	for a subset $M$ of a metric space $(X,d)$.
\end{definition}

\begin{theorem}[(Characterisation of ips via circumradius)]\label{chracterisationofipsviacircumradius}
	Let $(X,\norm{\cdot})$ be a normed vector space, $\dim X\geq 2$. The following items are equivalent
	\begin{enumerate}
		\item
			$(X,\norm{\cdot})$ is an inner product space,
		\item
			for all $x_{0}\in X$ and all $r>0$ we have $S(\partial B_{r}(x_{0}))=r$,
		\item
			there exists a $x_{0}\in X$ and a $r>0$, such that $S(\partial B_{r}(x_{0}))=r$.
	\end{enumerate}
\end{theorem}
\begin{proof}
	We start by proving (i)$\Rightarrow$(ii). If $(X,\norm{\cdot})$ is an inner product space and we take $x_{0},r,u,v,w$ as in (ii) then $S(\partial B_{r}(x_{0}))\leq r$
	by Lemma \ref{embeddingthreepointsonasphereandthecenter} and ``$=$'' by Lemma \ref{supremumofcircumradiusisattained}. The implication (ii)$\Rightarrow$(iii)
	is clearly true. Now assume (iii), i.e. there exists a sphere, which is not degenerate. Then by the contrapositive of Lemma \ref{allnonipsspheresaredegenerate}
	this implies (i).
\end{proof}

Taking the risk of repeating ourselves, we want to spell out, how one could go about in finding out if a normed space is an inner product space.

\begin{corollary}[(Computing a single $S(\partial B_{r}(x_{0}))$ decides about ips or nips)]
	Let $(X,\norm{\cdot})$ be a normed vector space, $\dim X\geq 2$ and $x_{0}\in X$, $r>0$. Then $S(\partial B_{r}(x_{0}))\geq r$ and
	\begin{itemize}
		\item
			if $r=S(\partial B_{r}(x_{0}))$ then $(X,\norm{\cdot})$ is an inner product space,
		\item
			if $r<S(\partial B_{r}(x_{0}))$ then $(X,\norm{\cdot})$ is a non inner product space.
	\end{itemize}
\end{corollary}
\begin{proof}
	Direct consequence of Theorem \ref{chracterisationofipsviacircumradius}.
\end{proof}

\begin{appendix}

\section{Appendix}

\begin{lemma}[(Intersection of sphere and straight line in inner product spaces)]\label{intersectionofsphereandline}
	In an inner product space $(V,\norm{\cdot})$ a sphere and a straight line can coincide in at most two points.
\end{lemma}
\begin{proof}
	Let $x_{0},a,v\in V$, $\norm{v}=1$, $r>0$. Consider the sphere $S\vcentcolon=\partial B_{r}(x_{0})$ and the straight line $L\vcentcolon=a+\R v$. We have $y\in S\cap L$
	if $y=a+tv$ for $t\in \R$ and $\norm{a+tv-x_{0}}=r$, which gives us
	\begin{align*}
		\norm{a+tv-x_{0}}^{2}=\langle (a-x_{0})+tv,(a-x_{0})+tv \rangle=\norm{a-x_{0}}^{2}+2\langle a-x_{0},v\rangle t+t^{2}.
	\end{align*}
	This is a quadratic polynomial, which has at most two solutions.
\end{proof}

\begin{remark}[(In nips lines and spheres can intersect in many points)]
	Considering the shape of $\partial B_{1}(0)$ in $(\R^{2},\norm{\cdot}_{\infty})$ we see that the conclusion of Lemma \ref{intersectionofsphereandline} does not hold
	if we drop the assumption, that $(V,\norm{\cdot})$ has a norm induced by an inner product.
\end{remark}

\begin{lemma}[(Circumradius of points on spheres in $(\R^{2},\norm{\cdot}_{\infty})$)]\label{circumradiusofspheresinnorminfty}
	We consider the Banach space $(\R^{2},\norm{\cdot}_{\infty})$. Let $x_{0}\in \R^{2}$ and $\rho>0$. For all $d\in (0,\infty]$ exist $u,v,w\in \partial B_{\rho}(x_{0})$,
	such that $r(u,v,w)=d$.
\end{lemma}
\begin{proof}
	Without loss of generality we can assume $x_{0}=0$. By scaling, i.e. $r(cu,cv,cw)=cr(u,v,w)$ for $c>0$, we only have to verify our claim for, say, $\rho=1$.\\
	\textbf{Step 1}
		Choose $u=(-1,1-s)$, $v=(-1+s,1)$ and $w=(-1,1)$ for $s\in (0,2)$. If we embed $u,v,w$ isometrically in the Euclidean plane we get an equilateral
		triangle with side lengths $s$, whose circumradius is given by
		\begin{align*}
			r(u,v,w)=\frac{s^{3}}{\sqrt{3s^{4}}}=\frac{s}{\sqrt{3}},
		\end{align*}
		so that we proved the proposition for $d\in (0,2/\sqrt{3})$.\\
	\textbf{Step 2}
		Choose $u=(-1,1-s)$, $v=(1,1-s)$ and $w=(0,1)$ for $s\in (1,2]$. By isometrically embedding these points in the Euclidean plane we get an isosceles triangle 
		with side lengths $s,s$ and $2$, such that
		\begin{align*}
			r(u,v,w)=\frac{2s^{2}}{\sqrt{(2+2s)\cdot 2\cdot 2 \cdot (-2+2s)}}=\frac{s^{2}}{2\sqrt{s^{2}-1}},
		\end{align*}
		which proves the proposition for $[2/\sqrt{3},\infty)$. The remaining case $r(u,v,w)=\infty$ was already treated in Example \ref{fourpointsnotembeddable}.
\end{proof}

\end{appendix}

\bibliography{smalllibrary.bib}{}
\bibliographystyle{amsalpha}
\bigskip
\noindent
\parbox[t]{.8\textwidth}{
Sebastian Scholtes\\
Institut f{\"u}r Mathematik\\
RWTH Aachen University\\
Templergraben 55\\
D--52062 Aachen, Germany\\
sebastian.scholtes@rwth-aachen.de}

\end{document}